\documentclass[hidelinks,onefignum,onetabnum]{siamart251216}
\usepackage{amsmath,amssymb}
\usepackage{hyperref}
\usepackage{algorithm}
\usepackage{algpseudocode}
\usepackage{cleveref}

\newtheorem{remark}{Remark}

\DeclareMathOperator*{\argmin}{arg\,min}
\DeclareMathOperator*{\rank}{rank}

\algnewcommand{\TRUE}{\textbf{true}}
\algnewcommand{\FALSE}{\textbf{false}}

\newcommand{\mGKS}{mGKS}
\newcommand{\R}{\mathbb{R}}
\newcommand{\Rmn}{\mathbb{R}^{m\times n}}
\newcommand{\Rnm}{\mathbb{R}^{n\times m}}
\newcommand{\Rnn}{\mathbb{R}^{n\times n}}
\newcommand{\Rnk}{\mathbb{R}^{n\times k}}

\newcommand{\Rmk}{\mathbb{R}^{m\times k}}
\newcommand{\Rkk}{\mathbb{R}^{k\times k}}

\newcommand{\C}{\mathcal{C}}
\newcommand{\J}{\mathcal{J}}
\newcommand{\B}{\mathcal{B}}
\newcommand{\hV}{\hat{V}}
\newcommand{\norm}[1]{\lVert #1 \rVert}
\newcommand{\smin}{\sigma_{\min}}

\newcommand{\fnorm}[1]{\lVert #1 \rVert_F}
\newcommand{\tnorm}[1]{\lVert #1 \rVert_2}
\newcommand{\mmn}{\min{(m,n)}}
\newcommand{\mc}{{:}}
\newcommand{\bigO}{\mathcal{O}}

\title{Computing Strong Rank-Revealing Factorizations for Matrices with Orthonormal Rows}
\author{Anil Damle\thanks{Department of Computer Science, Cornell University, Ithaca, NY 14853. (\email{damle@cornell.edu})\funding{AD is partially funded by the DOE Office of Science (DE-SC0025453).}}}
\date{\today}

\begin{document}
\maketitle

\begin{abstract}
We show that a pivoting strategy due to Stewart (based on work by Bischof) computes a strong rank-revealing factorization when applied to a matrix with orthonormal rows. When paired with the classical column selection algorithm of Golub, Klema, and Stewart (GKS) it helps achieve rank-$k$ approximation accuracy bounds and basis conditioning as good as those from applying a strong rank-revealing factorization directly to A. We then extend this framework in two directions: (1) providing analysis of GKS when only approximations of right singular vectors are available and (2) providing a randomized variant of the pivoting strategy for matrices with orthonormal rows that achieves the same theoretical guarantees but can return the desired subset two orders of magnitude faster than the deterministic variant.
\end{abstract}

\section{Introduction}
\label{sec:introduction}

Given a matrix $A\in\Rmn$ a common task is to compute $k$ ``good'' columns of $A$. Let $\C$ denote a subset of $[n]$ with $\lvert \C \rvert = k$. Two simple criteria we may seek to optimize are the error approximating $A$ by a rank-k matrix constrained to use $A(:,\C)$ as a basis for the column space, i.e.,
\begin{equation}
\label{eq:lrapprox}
	\|A-P_\C A\|_{\{2,F\}},
\end{equation}
where $P_\C$ is the orthogonal projector onto the span of $A(:,\C)$, and the quality of the basis, i.e.,
\begin{equation}
\label{eq:smin}
	\smin{(A(:,\C))}.
\end{equation}
The latter criterion is, e.g., particularly relevant when searching for basic solutions to least-squares problems. It is in this context that Golub, Klema, and Stewart introduced their scheme for selecting $\C$~\cite{golub_klema_stewart}; henceforth we refer to their scheme as GKS.

GKS proceeds in two stages. First, the leading $k$ right singular vectors of $A$ are computed and we denote the associated matrix $V_k\in\R^{n\times k}$. Then, a column-pivoted QR factorization~\cite{golub1965numerical,businger1965linear} is applied to $V_k^T$ to select $\C$. The first bounds for this method in terms of the quality of the pivoted QR factorization can be found in~\cite{hong_pan} for the two norm and~\cite{armstrong2025structure} for the Frobenius norm. The key point is that GKS reduces the column based low-rank approximation and basis conditioning problems to computing a rank-revealing QR factorization of a matrix with orthonormal rows.

When computing the necessary pivoted QR factorization of $V_k^T$ one has two choices. The common choice is an algorithm due to Golub and Businger~\cite{golub1965numerical,businger1965linear}. Implementations of this algorithm are widely available and the performance is often good in practice. However, its worst case theoretical performance is abysmal (with bounds that grow exponentially in $k$). In fact, recent work~\cite{chen2026iteris} provides pathological examples (for sufficiently large matrices) that exhibit poor behavior even under the orthonormal row constraint---addressing an open problem~\cite{Simons2025}. If good theoretical guarantees are desired, algorithms that compute strong rank-revealing QR factorizations (such as that of Gu and Eisenstat~\cite{gu_srrqr}) can be used. However, these algorithms can be much less efficient in practice (particularly as the desired bounds are strengthened).

A pivoting scheme due to Stewart~\cite{stewart1990incremental} (based on work by Bischof~\cite{bischof1990incremental}) is preferable to both options in this context. We refer to the scheme as Bischof-Stewart pivoting. In \cref{sec:BStheory} we show that when applied to a matrix with orthonormal rows Bischof-Stewart pivoting computes a strong rank-revealing QR factorization with stronger theoretical guarantees than provided by the Gu and Eisenstat algorithm when limited to polynomial computational complexity. Then, in \cref{sec:randBS} we provide a randomized variant of the pivoting strategy for matrices with orthonormal rows that provides the same theoretical guarantees while being much faster in practice. While the runtime is not deterministic, it strongly outperforms (e.g., by up to two orders of magnitude) the deterministic variant in all cases we consider.

After the introduction of necessary background material in~\cref{sec:background} and the analysis of Bischof-Stewart pivoting in \cref{sec:BStheory}, \cref{sec:mGKStheory} studies the effectiveness of GKS when augmented with Bischof-Stewart pivoting, which we refer to as modified GKS (\mGKS). In the $\tnorm{\cdot}$ norm, we show that for~\cref{eq:lrapprox} \mGKS\ is as good as a strong rank-revealing QR factorization (see, e.g.,~\cite{gu_srrqr}); the same is true for~\cref{eq:smin}. Specifically, we show that \mGKS\ selects $\C$ such that 
\[
	\tnorm{A-P_\C A} \leq \sqrt{1+k(n-k)}\sigma_{k+1}(A)
\]
and 
\[
	\smin(A(:,\C))\geq \frac{\sigma_k(A)}{\sqrt{1+k(n-k)}}.
\]
When $m\geq n$ this algorithm provides the same theoretical guarantees as recent work by Osinsky~\cite{osinsky2025close} in $\tnorm{\cdot}$ while being much more computationally efficient since excessive residual computations are avoided.\footnote{When $m<n$ Osinsky's algorithm admits better bounds since $n$ can be replaced by $\mmn$; \mGKS\ is oblivious to $m$ so while practical performance may be similar it is unlikely an equivalent bound could be produced.} In the $\fnorm{\cdot}$ \mGKS\ may (theoretically) lag behind Osinsky's algorithm for~\cref{eq:lrapprox}, but using results from~\cite{armstrong2025structure} we are able to characterize matrices where the bounds are similar (though, again, \mGKS\ is cheaper). 

To address situations where $V_k$ is not available and is, instead, approximated by some $\hV_k,$ \cref{sec:mGKStheory} highlights how in either the two or Frobenius norm \mGKS\ simply pays a sub-optimality factor of
\[
	\frac{\norm{A-A\hV_k\hV_k^T}_{\{2,F\}}}{\norm{A-AV_k V_k^T}_{\{2,F\}}},
\]
which naturally captures how much worse $\hV_k$ is at approximating $A$ than $V_k$. Notably, $\norm{A-A\hV_k\hV_k^T}_{\{2,F\}}$ provides a suitable target when there is no appreciable gap between $\sigma_k$ and $\sigma_{k+1}$ and approximation of the dominant subspace is ill-posed~\cite{drineas2019low}. This means (approximate) \mGKS\ is well suited to situations where $A$ is only available via matrix-vector products.\footnote{It also mirrors results in~\cite{osinsky2025close,cortinovis2026adaptive}, where the authors also analyze performance when an approximation of $V_k$ is used.} In this setting there are many ways to compute $\hV_k$ but there is no natural way to compute a strong rank-revealing QR factorization of $A$ directly.

Column subset selection is simultaneously a classical problem in numerical linear algebra, and one that merits, and receives, continued attention. \Cref{sec:discussion} highlights several ways in which the results herein can be extended (e.g., to two sided interpolatory factorizations), improved (e.g., to try and close the theoretical gap with volume-sampling), and provide natural improvements in application areas.

\subsection{Related work}
In the Frobenius norm, Cortinovis and Kressner~\cite{cortinovis2026adaptive} provide a randomized algorithm that, in expectation, achieves the same upper bound as Osinsky without residual computations. Leveraging the fact that the algorithm in~\cite{cortinovis2026adaptive} is equivalent to $k$-volume sampling of a projection matrix, Epperly~\cite{epperly2025adaptive} gives an independent proof of the same approximation bounds and provides more efficient sampling methods. The origins of these methods are in the (k-)volume sampling literature~\cite{deshpande_rademacher_projectiveclustering,deshpande2006adaptive,deshpande_rademacher_volumesampling,guruswami2012optimal,belhadji2020determinantal}, which we will not survey here.

This manuscript provides an improved understanding of how to leverage randomness within column selection algorithms. At one extreme sampling strategies can be highly efficient (e.g., uniform sampling), albeit with weak error guarantees. Conversely, computing a (full) strong rank-revealing factorization or performing ``exact'' k-volume sampling is quite expensive. A natural middle ground is to either approximate $V_k$ or ``sketch'' $A$ as $GA$ using some appropriate $G$ and then apply a column selection algorithm to $V_k^T$ or $GA$ (e.g., a simple pivoted QR). Recent work~\cite{dong2023simpler,dong2025robust} elaborates on this tradeoff. 

At least in the two norm, the most natural blend of theoretical guarantees and practical performance is to leverage randomness to approximate $V_k$ and then follow that up with (randomized) Bischof-Stewart pivoting applied to $V_k^T$. To get strong theoretical guarantees from $GA$ typically requires appealing to strong rank-revealing factorizations~\cite{martinsson_rokhlin_tygert_randid,grigori2025randomizedstrongrankrevealingqr}, which could be inefficient in practice. Random sampling schemes with strong error guarantees such as those highlighted earlier~\cite{cortinovis2026adaptive,epperly2025adaptive} provide no better guarantees than \mGKS\ in the two norm (regardless of whether true or approximate singular vectors are used). Moreover, their error guarantees are probabilistic (i.e., in expectation)---those of \mGKS\ are always satisfied, even when our randomized pivoting scheme is used. 

\section{Background}
\label{sec:background}

\subsection{Matrix Factorizations}
We will primarily leverage two key matrix factorizations: (partial) pivoted QR factorizations and the singular value decomposition (SVD). Given a permutation matrix $\Pi\in\Rnn$ and any $k\leq\mmn$ we let 
\begin{equation}
\label{eq:pivotedQR}
A\Pi = \begin{bmatrix}Q_1 & Q_2\end{bmatrix}\begin{bmatrix}R_{11} & R_{12} \\ & M\end{bmatrix}
\end{equation}
denote a partially computed (full) QR factorization of $A\Pi$ where $Q_1\in\Rmk$ and $Q_2\in\R^{m\times (m-k)}$ have orthonormal columns, $R_{11}\in\Rkk$ is upper triangular, $M\in\R^{(m-k) \times (n-k)}$, and $R_{12}\in\R^{k\times (n-k)}$. Observe that, since we place no structural restrictions on $M,$ we have defined a partial reduction to upper triangular form for $k<\mmn$; for this reason we use $M$ rather than $R_{22}$. Lastly, it will sometimes be convenient to split the permutation $\Pi$ as 
\[
	\Pi = \begin{bmatrix}\Pi_1 & \Pi_2\end{bmatrix},
\]
where $\Pi_1$ is the first $k$ columns of $\Pi$ and $\Pi_2$ the remaining $n-k$. When $k=\mmn$ we call this a pivoted QR factorization of $A$.

For the SVD, again given any $k\leq\mmn$ we denote
\begin{equation}
\label{eq:svd}
	A = U_k\Sigma_k V_k^T + U_\perp \Sigma_\perp V_\perp^T,
\end{equation}
where $U_k\in\Rmk,$ $V_k\in\Rnk,$ and $\Sigma_k\in\Rkk$ are used to denote the leading $k$ left singular vectors, right singular vectors, and singular values, and $U_\perp\in\R^{m\times (m-k)},$ $V_\perp\in\R^{n\times (n-k)},$ $\Sigma_\perp\in\R^{(m-k)\times (n-k)}$ denote the remaining singular values and vectors. As usual, $U_k,V_k,U_\perp,$ and $V_\perp$ have orthonormal columns and $\Sigma_k$ and $\Sigma_\perp$ are diagonal with non-negative entries denoted $\sigma_1,\ldots,\sigma_{\mmn}$.

\begin{remark}
In either of the preceding two factorizations, if $k=\mmn$ then any matrices with 0 dimension ``disappear.'' 
\end{remark}

\subsection{(Strong) rank-revealing QR factorizations}
It follows from~\cref{eq:pivotedQR} that $\norm{A-P_\C A} = \norm{M}$ for any unitarily invariant norm and $\smin(A(:,\C)) = \smin(R_{11})$. Therefore, the questions we asked in the introduction are intrinsically tied to pivoted QR factorizations. Accordingly, there has been significant work on algorithms that compute $\Pi$ to ensure that the QR factorization~\cref{eq:pivotedQR} has certain desirable properties. Of particular note is the development of (strong) rank-revealing QR factorizations (see, e.g.,~\cite{chandrasekaran_1994,hong_pan} for a discussion of rank-revealing and the refinement to strong rank-revealing in~\cite{gu_srrqr}).

For a factorization to be strong rank-revealing it must satisfy stringent criteria. Rather than just requesting that $\|M\|_2\approx \sigma_{k+1}$ and $\smin(R_{11}) \approx \sigma_k,$ the singular values of $R_{11}$ must well approximate $\sigma_1,\ldots,\sigma_k,$ the singular values of $M$ must well approximate $\sigma_{k+1},\ldots,\sigma_{\mmn},$ and the interpolation coefficients $(R_{11}^{-1}R_{12})_{ij}$ must have controlled size. Precise statements may be found in~\cite{gu_srrqr}; recently,~\cite{damle2025estimating} provided necessary and sufficient conditions for $\Pi$ to induce a strong rank-revealing factorization.\footnote{Once $\Pi$ has been fixed, $R_{11}$ is essentially unique; $M$ has more flexibility since it is not structurally constrained, but its singular values are determined by $\Pi$.}

This manuscript is primarily concerned with the reduced criteria in~\cref{eq:lrapprox} and~\cref{eq:smin}. Nevertheless, we will show that the modified pivoting strategy we discuss next computes a strong rank-revealing factorization in certain settings. Moreover, strong rank-revealing factorizations give good bounds for those problems and provide a natural point of comparison. From a computational standpoint, a key consideration is that algorithms for strong rank-revealing QR typically have to allow for ``swaps.'' In other words, for some fixed $k$ a candidate $\Pi_1$ is computed (e.g., via the standard pivoted QR algorithm of Golub and Businger~\cite{golub1965numerical,businger1965linear}) and then ``fixed'' by computing swaps between $\Pi_1$ and $\Pi_2$ till the desired bounds are achieved. Such a scheme is less efficient (albeit marginally so if the number of swaps is small) and a key consideration is how to bound the number of swaps---doing so necessitates slightly relaxing the most stringent criteria that could be asked for. For \mGKS\ we need no such relaxation.

\subsection{Stewart's modified pivoting criterion}
\Cref{secA:BSpivoting} contains a comprehensive description of the Bischof-Stewart pivoting scheme along with numerical experiments to illustrate practical performance. Here we briefly summarize the relevant ideas. Given a partially computed QR factorization as in~\cref{eq:pivotedQR} the core idea is to pick the next pivot (i.e., column $k+1$ of $\Pi$) by choosing the $j$ that minimizes
\begin{equation}
\label{eq:BSPivotBrief}
\frac{1+\|R_{11}^{-1}R_{12}e_j\|_2^2}{\|Me_j\|_2^2}.
\end{equation}
By convention, at the first step (i.e., $k=0$ to $k=1$) we simply choose the $j$ that maximizes $\|Me_j\|_2^2$ (i.e., the largest column of $A$). The standard Golub-Businger pivoting strategy simply sticks to this selection criterion at each step, always choosing $j$ to maximize $\|Me_j\|_2^2$.  

The motivation behind minimizing \cref{eq:BSPivotBrief} is to minimize the growth in $\|R_{11}^{-1}\|_F$. To see this, consider a partially computed QR factorization of $A\in\Rmn$ after $k \leq \mmn$ steps of the form~\cref{eq:pivotedQR}. If we let $\hat{R}_{11}$ denote the upper left $(k+1) \times (k+1)$ block of~\cref{eq:pivotedQR} after column $j$ of the block matrix 
\[
\begin{bmatrix} R_{12} \\ M\end{bmatrix}
\]
is chosen as a pivot (i.e., columns $k+1$ and $k+j$ of $A\Pi$ are swapped) and the factorization is progressed by one column (i.e., one additional Householder reflector is applied after the swap to reduce column $k+1$ of $A\Pi$ to upper triangular form), then
\begin{equation}
\label{eq:RinvF_update}
\|\hat{R}_{11}^{-1}\|_F^2 = \|R_{11}^{-1}\|_F^2 + \frac{1+\|R_{11}^{-1}R_{12}e_j\|_2^2}{\|Me_j\|_2^2}.
\end{equation}
This is a simple computation that follows from the fact that
\[
	\hat{R}_{11}^{-1} = \begin{bmatrix} R_{11}^{-1} & -\frac{R_{11}^{-1}R_{12}e_j}{\|Me_j\|_2} \\ & \frac{1}{\|Me_j\|_2} \end{bmatrix}.
\]

This small change has, in certain circumstances, profound impact on the theoretical performance. \Cref{subA:illustrative_performance} shows that such a scheme can be nearly as computationally efficient as the Golub-Businger scheme.\footnote{It achieves the same asymptotic complexity, but the floating point operation count is slightly higher.} Moreover, its form enables a fast randomized variant with strong theoretical guarantees. A natural question, therefore, is why the modified pivoting strategy didn't ``take off.'' While impossible to provide a definitive answer, one possible reason is that for certain pathological problems such as the Kahan matrix~\cite{kahan1966numerical} the Bischof and Stewart scheme also fails. Its value is only theoretically evident when applied to matrices with orthonormal rows---as in the GKS framework.

\section{Bischof-Stewart is strong rank revealing for matrices with orthonormal rows}
\label{sec:BStheory}

Given a matrix with orthonormal rows \cref{thm:BSpivot} characterizes how well the pivoting scheme of Bischof and Stewart controls singular values of the selected subset. Notably, Osinsky~\cite{osinsky2025close} proved the $i=m$ case in the analysis of that paper's main algorithmic contribution (which is distinct from \mGKS), and Avron and Boutsidis~\cite{avron2013faster} provide similar guarantees albeit with an algorithm that is quadratic in the number of columns (though that cost can be reduced~\cite{kozyrev2026subset}). \Cref{cor:osinsky} highlights the connection to Osinsky's work and that this algorithm achieves the quality of a maximum-volume submatrix~\cite{goreinov1997theory}, albeit without any guarantees of having maximal volume. For \cref{thm:BSpivot} we provide a direct proof based on \cref{eq:RinvF_update}, though aspects of the strategy mirror Osinsky's work.
\begin{theorem}
\label{thm:BSpivot}
Given $Z\in\Rnm$ with $Z^TZ=I,$ if Bischof-Stewart pivoting is applied to compute a pivoted QR factorization of $Z^T$ then the $m$ columns selected, denoted $\C,$ satisfy
\[
	\sigma_i(Z(\C,\mc))\geq \frac{1}{\sqrt{1+\frac{i(n-m)}{m-i+1}}}
\]
for $i=1,\ldots, m$.
\end{theorem}
\begin{proof}
We prove this bound progressively by tracking $\|R_{11}^{-1}\|_F^2$ while computing a pivoted QR factorization of $Z^T$ using Bischof-Stewart pivoting (following the notation of \cref{eq:pivotedQR} at each step). We will convert to a spectral norm bound later. We then use the fact that the singular values of $Z(\C,:)$ are the same as those of $R_{11}$ at the end of the factorization. 

To simplify notation, let $f_i = \|R_{11}^{-1}\|_F^2$ at step $i$. Because $Z^T$ has orthonormal rows there is at least one column with squared norm greater than or equal to $m/n$; since the pivoting scheme picks the largest column, we have that $f_1 \leq n/m$.

Using \cref{eq:RinvF_update} we can build the recurrence 
\[
f_{i+1} = f_i + \min_{j=1,\ldots,n-i} \frac{1+\|R_{11}^{-1}R_{12}e_j\|_2^2}{\|Me_j\|_2^2}.
\]
The entries in both the numerator and denominator of the minimization objective are nonnegative, and $\|Me_j\|_2^2 > 0$ for at least one $j$. So, we can bound the minimizer by the ratio of the sums as
\[
\min_{j=1,\ldots,n-i} \frac{1+\|R_{11}^{-1}R_{12}e_j\|_2^2}{\|Me_j\|_2^2} \leq \frac{(n-i) + \|R_{11}^{-1}R_{12}\|_F^2}{\|M\|_F^2}.	
\]
We now observe that $M\in\R^{(m-i)\times (n-i)}$ and $\begin{bmatrix}R_{11} & R_{12}\end{bmatrix}\in\R^{i\times n}$ have orthonormal rows. From this it follows that $\|M\|_F^2 = m-i$ and $\|R_{11}^{-1}R_{12}\|_F^2 = \|R_{11}^{-1}\|_F^2-i$. This yields the recurrence 
\[
	f_{i+1} \leq f_i + \frac{(n-i) + f_i - i}{m-i}.
\]
Grouping terms and forcing a common denominator we find that
\[
	f_{i+1} \leq \frac{m-i+1}{m-i}f_i + \frac{n-2i}{m-i}.
\]
We assert that $f_i \leq \frac{i(n-i+1)}{m-i+1}$ for $i=1,\ldots,m$ and perform a quick inductive proof; from above $f_1\leq n/m$. Running forward the inductive hypothesis yields
\begin{align*}
f_{i+1} &\leq \left(\frac{m-i+1}{m-i}\right)\frac{i(n-i+1)}{m-i+1} + \frac{n-2i}{m-i}\\
&\leq \frac{n+(n-1)i-i^2}{m-(i+1)+1}\\
&\leq \frac{(i+1)(n-(i+1)+1)}{m-(i+1)+1}.
\end{align*} 

To convert to a bound on $s_i = \|R_{11}^{-1}\|_2^2$ at step $i$ we use the fact that at each step $R_{11}$ is an $i\times i$ submatrix of a matrix with $i$ orthonormal rows---so, all of its singular values must be $\leq 1$. Therefore, $s_i \leq f_i - (i-1)$ yielding the bound
\begin{align}
s_i &\leq f_i - (i-1) \nonumber\\
&\leq \frac{i(n-i+1)}{m-i+1} - (i-1) \nonumber\\
&\leq \frac{i(n+1)-i^2+i^2-i(m+1)+(m-i+1)}{m-i+1} \nonumber\\
&\leq 1 + \frac{i(n-m)}{m-i+1}. \label{eqn:sbound}
\end{align}

Finally, we observe that once the factorization is complete, $s_i$ bounds the minimal singular value of the upper left $i\times i$ block of the $m\times m$ upper triangular matrix $R_{11}$ whose singular values match those of $Z(\C,:)$. The result then immediately follows from singular value interlacing. 
\end{proof}
\begin{corollary}[Theorem 5 in~\cite{osinsky2025close}]
\label{cor:osinsky}
Given $Z\in\Rnm$ with $Z^TZ=I,$ if Bischof-Stewart pivoting is applied to compute a pivoted QR factorization of $Z^T$ then the $m$ columns selected, denoted $\C,$ satisfy
\[
	\|Z(\C,\mc)^{-1}\|_2 \leq \sqrt{1+m(n-m)}
\]
and
\[
	\|Z(\C,\mc)^{-1}\|_F \leq \sqrt{m(n-m+1)}.
\]
\end{corollary}

\Cref{thm:BSpivot} shows that when applied to matrices with orthonormal rows Bischof-Stewart pivoting computes a strong rank-revealing QR factorization (per~\cite{gu_srrqr} and~\cite{damle2025estimating}) for $i = 1,\ldots,m$. To formalize this statement we extract from the proof of \cref{thm:BSpivot} that at step $i$ we have $\|R_{11}^{-1}R_{12}\|_{\max} \leq \sqrt{1 + \frac{i(n-m)}{m-i+1}},$ and observe that when $i<m$ $\sigma_{\ell}(M) = \sigma_{\ell+i}(Z^T)$. 

In fact, in the setting of orthonormal rows \cref{thm:BSpivot} shows that Bischof-Stewart pivoting is preferable to the strong rank-revealing algorithm of Gu and Eisenstat~\cite{gu_srrqr}. Bischof-Stewart yields stronger control on $\sigma_i(R_{11})$ for $i<m,$ and to get as strong of a result Algorithm 4 of~\cite{gu_srrqr} at $i=m$ would require setting $f=1$ (at which point the algorithm's runtime may not be polynomial in the size of the matrix). 
\begin{remark}
It is possible that this gap is an artifact of the analysis from~\cite{gu_srrqr} where the special case of orthonormal rows is not considered. Nevertheless, the potential for multiple swaps at step $k$ would significantly complicate per singular value analysis. Practically, the pivoted QR of Golub and Businger~\cite{businger1965linear} often suffices.
\end{remark}

\section{Theoretical guarantees for \mGKS}
\label{sec:mGKStheory}

We first give a brief description of the \mGKS\ framework and then provide a theoretical treatment. To accomplish this we provide a generic result that highlights how well columns of $A$ selected from an arbitrary matrix with orthonormal rows can be used to approximate $A$. We then incorporate the results from \cref{sec:BStheory} to provide end to end guarantees for \mGKS.

\subsection{The \mGKS\ framework}
We briefly summarize \mGKS\ as a framework for selecting columns of a matrix in \cref{alg:mGKS}. It is deliberately written to return $\C \subset [n]$ of cardinality $k$ representing the selected columns of $A$. These columns can then be used as needed, e.g., to compute a low-rank approximation.  

\begin{algorithm}
\caption{\mGKS}
\label{alg:mGKS}
\begin{algorithmic}[1]
\Require matrix $A \in \Rmn$ and target number of columns $k \leq \rank{(A)}$.
\State Compute $\hV_k,$ an estimate of $V_k$ with $\hV_k^T\hV_k=I$ \Comment{Estimate the leading $k$ right singular vectors of $A$}
\State Compute the column pivoted QR factorization
\[
	\hV_k^T\begin{bmatrix} \Pi_1 & \Pi_2\end{bmatrix} = Q_1 \begin{bmatrix} R_{11} & R_{12}\end{bmatrix}.
\]
using the Bischof-Stewart pivoting scheme. \Comment{See \cref{secA:BSpivoting} for details}
\State Let $\C$ denote the $k$ rows of $\Pi_1$ that contain non-zero entries.
\State \Return $\C$
\end{algorithmic}
\end{algorithm}

The rationale for framing \mGKS\ as a framework is that there are lots of choices for how to compute $\hV_k$. The details of which scheme is used are inconsequential for the theoretical bounds, all that matters is how well $\hV_k$ can be used to approximate $A$ (relative to $V_k$). Natural choices are to use a randomized SVD~\cite{halko_finding_structure_with_randomness} (as in~\cite{armstrong2025structure}) or a sketch of $A$ whose rows are then orthogonalized (as in~\cite{martinsson_rokhlin_tygert_randid}); other choices such as Krylov subspace based methods are also good.

\subsection{Approximation using an arbitrary subspace}

Given an arbitrary matrix $Z\in\Rnk$ (here, one should think of it as an approximation of $V_k$ as in \cref{alg:mGKS}), how well can we approximate $A$ using columns selected from $Z^T$? We address this question in two stages. First, in \cref{thm:Z2F} we prove a generalization of Theorem 9.1 in~\cite{halko_finding_structure_with_randomness} that allows us to incorporate the overlap between a sampling matrix and an arbitrary orthogonal matrix (as opposed to just the right singular vectors); this result is of independent interest. We then specialize the result in \cref{cor:Z2,cor:ZF} to sampling with a subset of the identity matrix (i.e., column selection)---this recovers lemma 4.1 from~\cite{sorensen_embree_deim} in the two norm case and yields new bounds in the Frobenius norm.\footnote{The proof of \cref{thm:Z2F} is motivated by observations from~\cite{sorensen_embree_deim}.} 

\begin{theorem}
\label{thm:Z2F}
Consider $A\in\Rmn,$ an arbitrary $Z\in\Rnk$ with $k<\mmn$ such that $Z^TZ = I,$ any $Z_{\perp} \in \R^{n\times (n-k)}$ such that $\begin{bmatrix} Z & Z_{\perp}\end{bmatrix}$ is orthogonal, and a sampling matrix $\Omega\in\R^{n\times \ell}$ with $\ell \leq n$. If $P_Y$ is the orthogonal projector onto the range of $Y=A\Omega$ then if $W_1 = Z^T\Omega$ has full row rank
\[
	\|A-P_Y A\|_{\{2,F\}}^2 \leq \|A(I-ZZ^T)\|_{\{2,F\}}^2 + \|A(I-ZZ^T)\|_{2}^2\|W_2W_1^{\dagger}\|_{\{2,F\}}^2,
\]
where $W_2 = Z_{\perp}^T\Omega$.
\end{theorem}
\begin{proof}
Since $P_Y A$ is the closest matrix to $A$ whose columns are in the range of $Y$ in both the two and Frobenius norm we have that
\[
	\|(I-P_Y)A\|_{\{2,F\}}^2 \leq \|A - YW_1^{\dagger}Z^T\|_{\{2,F\}}^2.
\]
I.e., we compare the best approximation to $A$ with range $Y$ to a constrained one constructed using the oblique projector $\Omega W_1^{\dagger}Z^T,$ which forces the approximation to have row space $Z$ and range contained in $Y$. 

Now, let $E = A - YW_1^{\dagger}Z^T$ and observe that 
\begin{align*}
	E\begin{bmatrix}Z & Z_{\perp} \end{bmatrix} &= \begin{bmatrix} AZ - YW_1^{\dagger} & AZ_{\perp}\end{bmatrix}\\
	&= \begin{bmatrix} A(Z - \Omega W_1^{\dagger}) & AZ_{\perp}\end{bmatrix}\\
	&= \begin{bmatrix} A(ZZ^T + Z_{\perp}Z_{\perp}^T)(Z - \Omega W_1^{\dagger}) & AZ_{\perp}\end{bmatrix}\\
	&= \begin{bmatrix} -AZ_{\perp} W_2 W_1^{\dagger} & AZ_{\perp}\end{bmatrix},
\end{align*}
where we have used that $ZZ^T(Z - \Omega W_1^{\dagger}) = 0$. We can now bound $\|A - YW_1^{\dagger}Z^T\|_{\{2,F\}}^2$ as
\[
	\|A - YW_1^{\dagger}Z^T\|_{\{2,F\}}^2 \leq \|AZ_{\perp}W_2W_1^{\dagger}Z^T\|_{\{2,F\}}^2 + \|AZ_{\perp}\|_{\{2,F\}}^2.
\]
We can factor out $\|AZ_{\perp}\|_2^2$ from the first term while retaining the bound; recognizing that $\|AZ_{\perp}\|_{\{2,F\}} = \|A(I-ZZ^T)\|_{\{2,F\}}$ completes the proof.
\end{proof}

We now refine \cref{thm:Z2F} when $\Omega$ is $k$ columns from the identity matrix, specifically those chosen by a pivoted QR factorization applied to $Z^T$ denoted 
\begin{equation}
\label{eqn:ZpQR}
Z^T\begin{bmatrix} \Pi_1 & \Pi_2\end{bmatrix} = Q_1 \begin{bmatrix} R_{11} & R_{12}\end{bmatrix}.
\end{equation}
We first consider the two norm and then the Frobenius norm.
\begin{corollary}[Lemma 4.1 in~\cite{sorensen_embree_deim}]
\label{cor:Z2}
Consider $A\in\Rmn$ and an arbitrary $Z\in\Rnk$ such that $Z^TZ = I$ with $k < \mmn$. Given a pivoted QR of $Z^T$ as in~\cref{eqn:ZpQR} with nonsingular $R_{11}$ we have that
\[
	\|A-P_{\C} A\|_2 \leq \frac{1}{\smin{(R_{11})}}\|A(I-ZZ^T)\|_2,
\]
where $P_\C$ is the orthogonal projector onto $A\Pi_1$.
\end{corollary}
\begin{proof}
We start with \cref{thm:Z2F} and factor out $\|A(I-ZZ^T)\|_2^2$ and observe that $W_1 = Z^T\Pi_1$ and $W_2 = Z_{\perp}^T\Pi_1$. Using a CS decomposition of $\begin{bmatrix}Z & Z_{\perp}\end{bmatrix}^T\begin{bmatrix} \Pi_1 & \Pi_2\end{bmatrix}$ we can conclude that $\|W_2W_1^{-1}\|_2 = \|R_{11}^{-1}R_{12}\|_2$. The same CS decomposition also shows that $\sqrt{1+\|R_{11}^{-1}R_{12}\|_2^2} = \frac{1}{\smin{(R_{11})}}$ from which the result follows. 
\end{proof}

\begin{corollary}
\label{cor:ZF}
Consider $A\in\Rmn$ and an arbitrary $Z\in\Rnk$ such that $Z^TZ = I$ and $k < \mmn$. Given a pivoted QR of $Z^T$ as in~\cref{eqn:ZpQR} with nonsingular $R_{11}$ we have that
\[
	\|A-P_{\C} A\|_F \leq \|A(I-ZZ^T)\|_F\sqrt{1+\frac{1}{r_k}\|R_{11}^{-1}R_{12}\|_F^2},
\]
where $P_\C$ is the orthogonal projector onto $A\Pi_1$ and
\[
r_k = \frac{\|A(I-ZZ^T)\|_F^2}{\|A(I-ZZ^T)\|_2^2}.
\]
\end{corollary}
\begin{proof}
We start with \cref{thm:Z2F} and factor out $\|A(I-ZZ^T)\|_F^2$ and observe that $W_1 = Z^T\Pi_1$ and $W_2 = Z_{\perp}^T\Pi_1$. Using a CS decomposition of $\begin{bmatrix}Z & Z_{\perp}\end{bmatrix}^T\begin{bmatrix} \Pi_1 & \Pi_2\end{bmatrix}$ we can conclude that $\|W_2W_1^{-1}\|_F = \|R_{11}^{-1}R_{12}\|_F$; the result follows by observing that $r_k = \|A(I-ZZ^T)\|_F^2/\|A(I-ZZ^T)\|_2^2$. 
\end{proof}

\Cref{cor:Z2,cor:ZF} show how the overall approximation error can be decomposed into two parts: (1) a part dependent on the approximating subspace and (2) a part due to the column selection scheme applied to $Z^T$. Methods to build the approximating subspace $Z$ are well studied (and continue to receive attention), so we do not dwell on them here.

\subsection{Theoretical guarantees of \mGKS}

In \cref{thm:mGKS} we collect the above results to provide complete control of the approximation efficacy of \mGKS\ in terms of the quality of $\hV_k$. The theorem deliberately highlights the suboptimality factor of the approximation relative to the SVD. The special case $\hV_k = V_k$ simplifies the bounds to make them more directly comparable to certain prior work.

\begin{theorem}
\label{thm:mGKS}
Given $A\in\Rmn$ for any $k < \rank{A}$ \cref{alg:mGKS} computes a column subset $\C$ that satisfies
\[
	\|A-P_\C A\|_2 \leq \eta_2\sqrt{1+k(n-k)}\sigma_{k+1}(A)
\]
and
\[
	\|A-P_\C A\|_F \leq \eta_F\sqrt{1+\frac{k(n-k)}{r_k}}\left(\sum_{i=k+1}^{\mmn} \sigma_{i}(A)^2\right)^{1/2},
\]
where 
\begin{equation}
\label{eqn:consts}
r_k = \frac{\|(A(I-\hV_k \hV_k^T))\|_F^2}{\|(A(I-\hV_k \hV_k^T))\|_2^2}, \quad
\eta_2 = \frac{\|A(I-\hV_k\hV_k^T)\|_2}{\|A(I-V_kV_k^T)\|_2}, \quad 
\text{and} \quad 
\eta_F = \frac{\|A(I-\hV_k\hV_k^T)\|_F}{\|A(I-V_kV_k^T)\|_F}.
\end{equation}
\end{theorem}
\begin{proof}
The result follows immediately from inserting bounds from \cref{thm:BSpivot} (and, in the case of the Frobenius norm, its proof) into \cref{cor:Z2,cor:ZF}. 
\end{proof}

\begin{remark}
In the Frobenius norm bound of \cref{thm:mGKS} we see an interesting phenomenon: the upper bound improves as the stable rank of the residual $A(I-\hV_k \hV_k^T)$ grows. The cleanest way to interpret the result is to consider the case where $\hV_k = V_k$. In that case $r_k$ is the stable rank of $\Sigma_{\perp}$. If the singular values of the residual $\Sigma_{\perp}$ are completely flat and $n \leq m$ the Frobenius norm upper bound in \cref{thm:mGKS} becomes $\sqrt{1+k}\|\Sigma_{\perp}\|_F$ and matches the best possible result. While a somewhat contrived situation, we often attempt to choose $k$ at an ``elbow'' in the singular values such that the singular values of $\Sigma_{\perp}$ are somewhat flat (e.g., at a noise floor).
\end{remark}

While the aforementioned results are quite strong, it is worth highlighting that even when $\hV_k^T\neq V_k$ is used to select columns $\C$ the bound
\[
	\|A-P_\C A\|_2 \leq \frac{\sigma_{k+1}(A)}{\smin(V_k(\C,:))}
\] 
always holds. While $\smin(V_k(\C,:))$ cannot be easily computed without access to $V_k$ it is possible this upper bound is smaller than the one from \cref{cor:Z2}. This suggests that \cref{thm:mGKS} could be quite loose in certain situations. The key question to address is when does picking columns from $\hV_k^T$ ensure good columns of $V_k^T$ are selected.\footnote{In fact, while not expected, it is possible that in certain cases $\smin(V_k(\C,:)) > \smin(\hV_k(\C,:))$.} Addressing this question directly is quite challenging.

An important aspect of \cref{thm:mGKS} is that it is near optimal in some sense. Specifically, in a slightly limited setting ($m=n+1$)~\cite{boutsidis2014near} shows that there is a matrix $A$ such that any selection of $k$ columns yields an approximation at least a factor of $\sqrt{n/k}$ worse than optimal. Unlike in the Frobenius norm, it is impossible to have a bound independent of $n$.

\subsection{Basis quality}
We have thus far focused on approximation error. When true right singular vectors are used (i.e., $\hV_k = V_k$) it is immediate (see, e.g., \cite{golub_klema_stewart,hong_pan}) that 
\[
	\smin(A(\mc,\C))\geq \sigma_{k}(A)\smin{(V_k(\C,\mc))}.
\]
Using Bischof-Stewart pivoting then implies that 
\[
\smin(A(\mc,\C))\geq \sigma_{k}(A)/\sqrt{1+k(n-k)}. 
\]

It is natural to consider what happens when an approximation to $V_k$ is used. A brief computation shows that for an arbitrary $\hV_k$ used within \mGKS\ we have that
\[
	\smin(A(\mc,\C))\geq \smin{(A\hV_k)}\smin{(\hV_k(\C,\mc))} - \eta_2 \sigma_{k+1}(A),
\]
where $\eta_2$ is the same as in \cref{eqn:consts}. Specifically, using Weyl's inequality we have that
\begin{align*}
	\smin(A(\mc,\C)) &= \smin(A\hV_k(\hV_k(\C,\mc))^T + A(I-\hV_k\hV_k^T)I(\mc,\C)) \\
	&\geq \smin(A\hV_k)\smin(\hV_k(\C,\mc)) - \|A(I-\hV_k\hV_k^T)\|_2 \\
	&\geq \smin(A\hV_k)\smin(\hV_k(\C,\mc)) - \eta_2 \sigma_{k+1}(A).
\end{align*}
Notably, as $\hV_k$ fails to capture the dominant right singular subspace of $A$ the quality of the basis may degrade accordingly. In extreme cases this bound could even become vacuous, particularly if there is not a gap between $\sigma_k(A)$ and $\sigma_{k+1}(A)$.

\section{Randomized Bischof-Stewart pivoting for matrices with orthonormal rows} 
\label{sec:randBS}

A byproduct of the proof of \cref{thm:BSpivot} is a clear path to a highly efficient randomized variant of Bischof-Stewart pivoting that provides a set of columns with the same theoretical guarantees as the deterministic scheme when applied to matrices with orthonormal rows---not just in expectation, but every time. This method is best suited to the case where $m\ll n,$ the typical \mGKS regime.

\subsection{The algorithm}
The key observation from \cref{thm:BSpivot} is that to satisfy the overall bounds on growth of $\|R_{11}^{-1}\|_{\{2,F\}}$ (and the individual singular value bounds) we only need to ensure that at each step
\begin{equation}
\label{eqn:RandAccept}
f_{i+1} \leq \frac{m-i+1}{m-i}f_i + \frac{n-2i}{m-i},
\end{equation}
where, as before, $f_i = \|R_{11}^{-1}\|_F^2$ when $R_{11}$ is $i\times i$ and $f_0=0$. When $A$ has orthonormal rows we are guaranteed an index $j$ in \cref{eq:BSPivotBrief} that satisfies this criterion. However, there are likely many such indices and we need not worry about finding the minimizer. Any index that satisfies the growth bound suffices.
\begin{remark}
In practice, one might want to add a small ``slack'' factor to the right hand side of \cref{eqn:RandAccept} (a small multiple of machine precision) to account for rounding errors when the bound is saturated by even the best column. In addition, one may be concerned that \cref{eqn:RandAccept} could try and enforce a reduction in $f_i$ when $n<2m$ as the second term can become negative. This is disallowed by the lower bound $f_i\geq i$ when $A$ has orthonormal rows.
\end{remark}

The natural way to leverage this observation is to avoid searching over all columns for a suitable index $j$ to use as the next pivot and, crucially, to not maintain ``state'' for all the columns (the expensive part). In the context of \cref{eq:pivotedQR} this means that we only compute $R_{11}$ and a representation of $\Pi$. At any given step only certain columns of $R_{12}$ and $M$ are constructed. It is important to note that this means the randomized algorithm, by default, only returns $\Pi,$ $Q,$ and $R_{11}$. If desired, $Q^T$ can be used to compute $R_{12}$. 

Building off this observation, the randomized algorithm proceeds in the following manner: we randomly sample $k_b$ columns from $A,$ update them to the current state using saved Householder reflectors from earlier steps, start selecting columns from the block progressively using Bischof-Stewart pivoting until we are unable to find one that satisfies \cref{eqn:RandAccept}, sample a new set of $k_b$ columns and continue till we have selected $m$ total columns. This process is encoded in \cref{alg:randBSQR} and we entitle the algorithm \texttt{randBSQR}. By construction, the selected subset of columns satisfies the same bounds as the deterministic variant; this is encapsulated in \cref{thm:randBSpivot}. For simplicity the algorithm is written to always run for $m$ steps. 

\begin{remark}
\Cref{thm:BSpivot} guarantees that at each step there is a pivot choice that satisfies \cref{eqn:RandAccept}. However, we could instead simply enforce that $f_i \leq \frac{i(n-i+1)}{m-i+1}$ (which is also always possible) and still get the same bounds on $\|R_{11}^{-1}\|_{\{2,F\}}$ when the algorithm terminates. The practical effect is that ``slack'' in earlier steps (i.e., a gap between the chosen pivot and the necessary per-step upper bound) could be spent later to, potentially, find an allowable pivot faster. The potential downside is slight degradation in the quality of columns selected relative to enforcing \cref{eqn:RandAccept} at each step.
\end{remark}

\begin{algorithm}[t]
\caption{Randomized Bischof-Stewart column selection.}
\label{alg:randBSQR}
\begin{algorithmic}[1]
\Require $A\in\mathbb{R}^{m\times n}$ with $AA^T = I$; block size $k_b$ (default $k_b=m$); sampling
         weights $g_j$ with $g_j>0$ if $\|A(\mc,j)\|_2>0$ (default $g_j=\|A(\mc,j)\|_2^2/m$)
\Ensure Index set $\C$ with $\lvert \C\rvert=m$ satisfying \cref{thm:randBSpivot}; $Q$ and $R$ such that $A(\mc,\C) = QR$
\State Initialize $\C = \begin{bmatrix}\phantom{1} \end{bmatrix}$;\quad $\J = \begin{bmatrix}1,2,\ldots,n\end{bmatrix}$;\quad  $f = 0$;\quad $R = \begin{bmatrix}\phantom{1} \end{bmatrix}$;\quad $i=0$
\While{$\lvert\C\rvert < m$}
	\State $\hat{\J} = \J\setminus \C$;\quad ColSelect$\gets$ \FALSE
	\While{ColSelect is \FALSE}
	  \State $\B\gets\Call{SampleBlock}{\hat{\J},\min{(k_b,\lvert\hat{\J}\rvert)},g}$ 
	  \State $\mathcal{S} = \begin{bmatrix}1,2,\ldots,\lvert \B\rvert\end{bmatrix}$ \Comment{Track active columns}
	  \State $X =$ \Call{Update}{$A(\mc,\B),\{v_p\}_{p=1}^i$} \Comment{$X = Q^TA(\mc,\B)$ using Householder reflectors}  
	  \For{$j\in\mathcal{S}$}
	  \State $\rho_j^2 = \|X(i+1\mc m,j)\|_2^2$;\quad $Rw_j = X(1\mc i,j)$ \Comment{If $i=0$ $w_j$ is empty and $\|w_j\|_2^2 = 0$}
	  \EndFor
	  \If{$\max_j \rho_j = 0$} \Comment{Sampled a zero block; sample new columns}
		  \State $\hat{\J} = \hat{\J}\setminus \B$ 
	  \EndIf
	    \While{$\lvert\C\rvert < m$, $\mathcal{S} \neq \emptyset$, and $\max_{j\in\mathcal{S}} \rho_j > 0$}
	    \State Choose $\displaystyle \ell \in \argmin_{j\in\{\mathcal{S}\vert \rho_j > 0\}} \frac{1+\|w_j\|_2^2}{\rho_j^2}$ and let $c = \frac{1+\|w_\ell\|_2^2}{\rho_\ell^2}$
	    \If{$f+c > \frac{m-i+1}{m-i}f + \frac{n-2i}{m-i}$} \label{line:fcheck}\Comment{\cref{eqn:RandAccept} violated; sample new columns}
	    	\State $\hat{\J} = \hat{\J}\setminus \B$ 
	    	\State \textbf{break} 
	    \Else \Comment{\cref{eqn:RandAccept} satisfied; append column to $\C$}	 
	    	\State $\C = \begin{bmatrix}\C & \B(\ell)\end{bmatrix}$;\quad $i=i+1$    	
	    	\State $v_i =$ \Call{house}{$X(i\mc m,\ell)$} \Comment{$(I-2v_iv_i^T)X(i\mc m,\ell) = \pm \|X(i\mc m,\ell)\|_2e_1$}
	    	\State $X(i\mc m,\mc) = X(i\mc m,\mc) - 2v_i(v_i^TX(i\mc m,\mc))$
	    	\For{$j\in\mathcal{S}$}
	    	\State $\rho_j^2 = \rho_j^2 - X(i,j)^2$
			\State $\displaystyle w_j = \begin{bmatrix} w_j - \alpha_j w_{\ell}\\ \alpha_j\end{bmatrix},$ where $\alpha_j = X(i,j)/X(i,\ell)$
			\EndFor
			\State $\displaystyle R = \begin{bmatrix} R & X(1\mc i-1,\ell)\\ 0 & X(i,\ell)\end{bmatrix}$ 
			\State $\mathcal{S} = \mathcal{S}\setminus \ell$ \Comment{Remove selected column from the active set}
			\State f = f + c \label{line:fupdate}\Comment{Update $\|R_{11}^{-1}\|_F^2$}
	    	\State ColSelect$\gets$ \TRUE
	    \EndIf
  		\EndWhile
  	\EndWhile
\EndWhile
\State \Return $\C$, $R$, and $\{v_i\}_{i=1}^m$ representing $Q$
\end{algorithmic}
\end{algorithm}

\begin{theorem}
\label{thm:randBSpivot}
Given $Z\in\Rnm$ with $Z^TZ=I,$ if \cref{alg:randBSQR} is used to select $m$ columns of $Z^T$ then those columns selected, denoted $\C,$ satisfy
\[
	\sigma_i(Z(\C,\mc))\geq \frac{1}{\sqrt{1+\frac{i(n-m)}{m-i+1}}}
\]
for $i=1,\ldots, m$.
\end{theorem}
\begin{proof}
The result follows the control over the growth of $f_i$ ($f$ in \cref{alg:randBSQR}). Specifically, \cref{line:fcheck} ensures that \cref{eqn:RandAccept} and therefore \cref{eqn:sbound} are satisfied; \cref{line:fupdate} makes $f$ track $\|R_{11}^{-1}\|_F^2.$ Because the sampling weights are non-zero for non-zero columns and rejected columns are not resampled till at least one new column has been found, we will find a suitable column (whose existence is guaranteed because $A$ has orthonormal rows; see the proof of \cref{thm:BSpivot}) in finitely many steps. 
\end{proof}

\subsection{Experimental setup}
We provide an optimized implementation of \cref{alg:randBSQR} as C++ MEX file usable from MATLAB~\cite{matlab}.\footnote{The implementation contains a few minor deviations from the stylized algorithm presented here, such as safeguards when using norm updating formula (see \cref{secA:BSpivoting}, additional blocking for computational efficiency, and better management of the Householder reflectors using WY form.} The implementations, tests, and code to reproduce figures are available at \url{https://github.com/asdamle/BischofStewartQR}. All experiments were performed on a MacBook Air with an M3 processor and 16 GB of RAM using the Apple Accelerate framework. Timings are computed using \texttt{timeit} to ensure stable results.

\subsection{Matrices}
Throughout these experiments we will primarily use three different test matrices with orthogonal rows. All originate from an $m\times n$ matrix $G$ with iid $\mathcal{N}(0,1)$ entries. The \emph{gaussian} matrices are constructed by orthogonalizing the rows. To build \emph{spiked leverage} matrices and \emph{needle} (simulating a "needle in a haystack" problem) we first weight the columns before orthogonalizing the rows. In the spiked leverage case we place a weight of 100 on $\lceil 3m/2\rceil$ columns and 1 on the remainder. In the needle example we place a weight of 1 on $\lceil 5m/4\rceil$ columns and $10^{-6}$ on all others. 

A few more test matrices appear sporadically. The \emph{graded leverage} matrices start with $G$ as above and then weight the columns with logarithmically spaced weights between 1 and $10^{-3}$ before orthogonalizing the rows. The \emph{coherent} matrices are built by picking $\lceil m/4 \rceil$ random vectors in $\R^m$ and then building a matrix with $n$ columns by uniformly at random sampling one of those vectors to produce each column of $A,$ and adding $5\%$ noise to each column and then orthogonalizing the rows. The \emph{collinear cluster} matrices again start from $G,$ but replace the first 20\% of the columns with a large fixed vector $v_0$ (chosen at random) plus a small amount of noise, and then orthogonalize the rows. Finally, the \emph{chebyshev} matrices sample the first $m$ Chebyshev polynomials at $n$ random points on $[-1,1]$ and then orthogonalizing the rows.

\subsection{Parameter choices}
There are a few parameter choices worth discussing. Selecting $k_b$ columns at a time yields practical efficiency, as does allowing for the selection of multiple columns out of each block. However, there are some points of tension highlighted by \cref{fig:kb}. First, we see that there is a trade-off between efficiency and $\|R_{11}^{-1}\|_F$. While the bound on $\|R_{11}^{-1}\|_F$ is always satisfied, in some cases there is a small benefit to picking larger blocks. From an efficiency perspective the optimal block size is around $k_b = 32,$ though that varies slightly with $m$. If $k_b$ is too large we end up sampling many more columns in total. Practically, we find $k_b = m$ to provide a good balance between performance and conditioning. There is also a choice in terms of the sampling distribution for columns. Sampling proportional to the initial squared column norms (i.e., $\|A(\mc,j)\|_2^2$) is a good choice and only adds a one time pre-processing step of $\bigO(mn)$. We refer to this scheme as norm-weighted sampling. In certain circumstances uniform sampling could be viable, but in adversarial settings it leads to sampling many more columns than necessary as illustrated in \cref{fig:sampling}.

\begin{figure}[t]
\centering
\includegraphics[width=\linewidth]{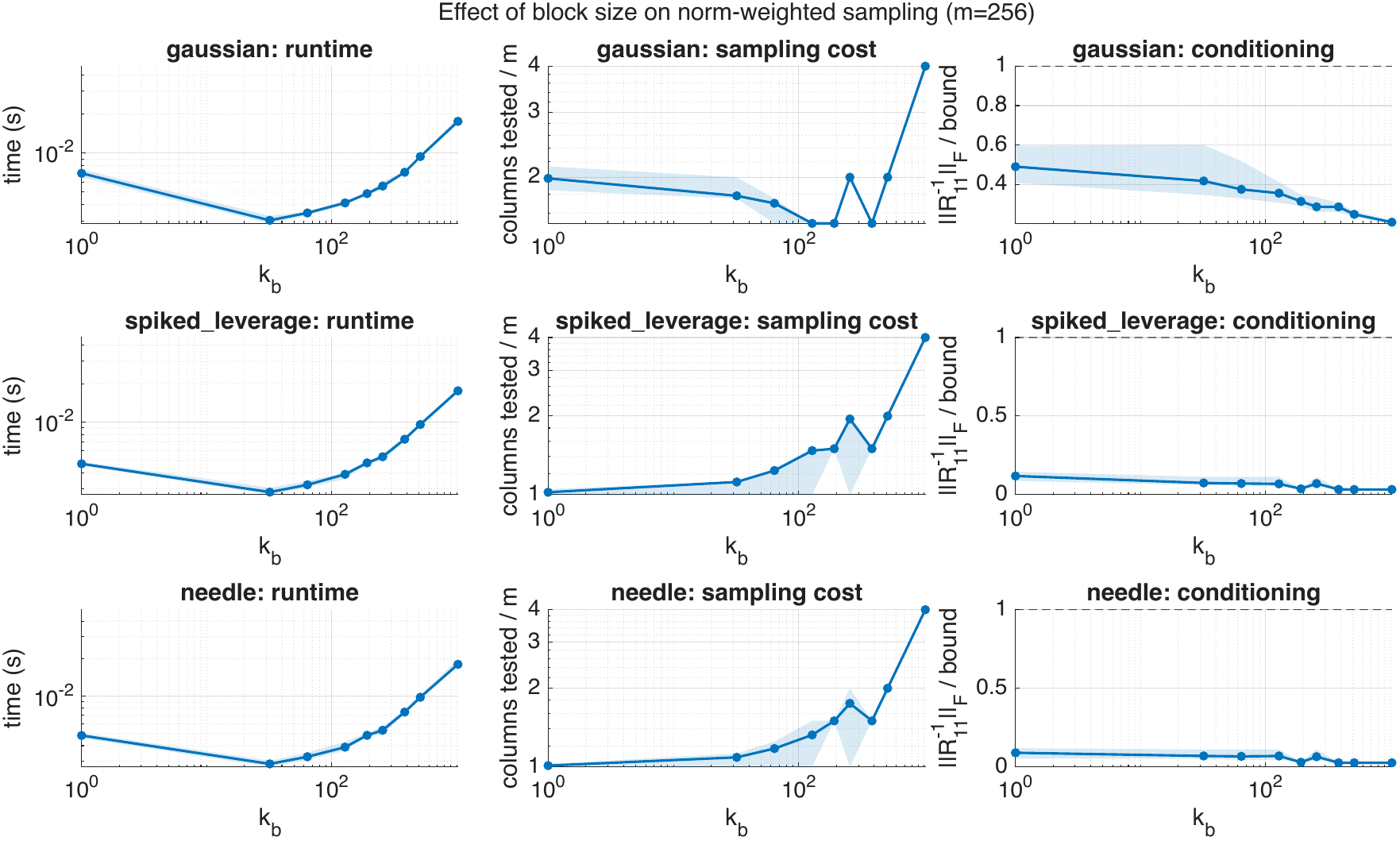}
\caption{When using norm-weighted sampling: time taken to compute the column subset (left), total columns sampled (middle) and $\|R_{11}^{-1}\|_F$ relative to the bound from \cref{cor:osinsky} (right) as a function of block size $k_b$. Each row corresponds to a different test matrix with orthonormal rows. $\|R_{11}^{-1}\|_F$ improves slightly with increased block size, albeit with a corresponding increase in time. The total number of sampled columns is always a small multiple of $m$. Data points are medians over 20 instances of the random matrix and shaded regions indicate min/max values.}
\label{fig:kb}
\end{figure}

\begin{figure}[t]
\centering
\includegraphics[width=\linewidth]{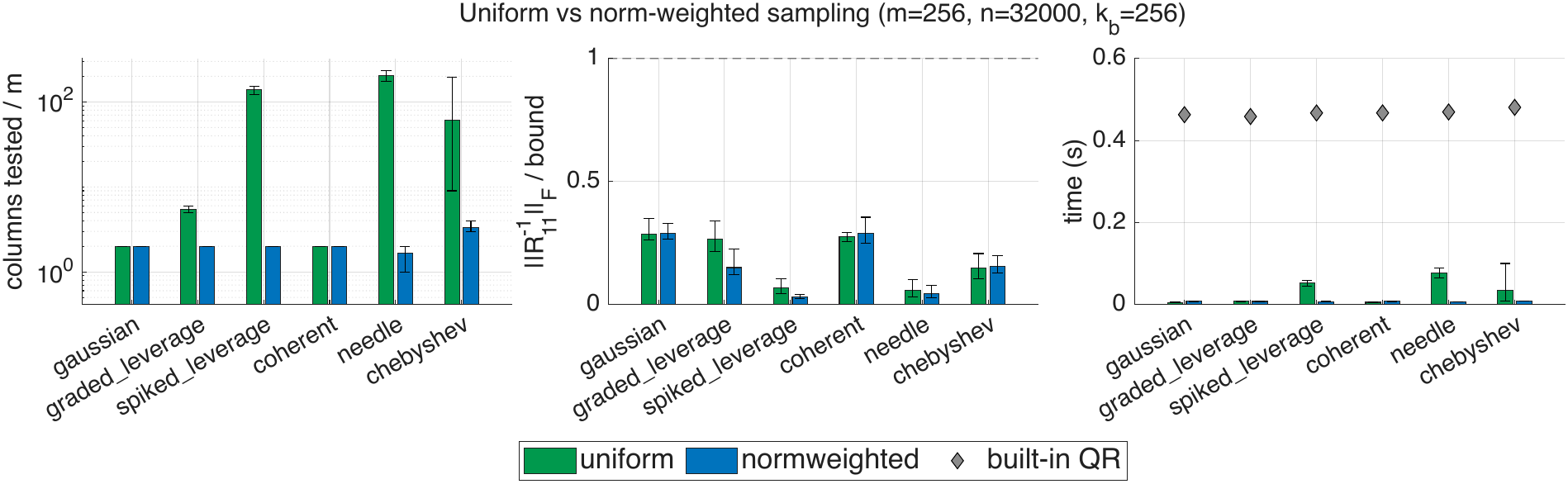}
\caption{Number of columns sampled (left), $\|R_{11}^{-1}\|_F$ (middle) and time taken to compute the column subset (right) for randBSQR with $k_b=m$ when either uniform or norm-weighted sampling is used. Bars represent medians over 20 trials with each matrix type and whiskers represent min/max observations over those trials. Norm-weighted sampling consistently results in $T_s$ being a small constant; uniform sampling is fine in some cases but catastrophic in others.}
\label{fig:sampling}
\end{figure}

\subsection{Computational cost}
The deterministic variant of Bischof-Stewart pivoting has the same asymptotic cost\footnote{For simplicity here we will always assume the algorithm is run for $m$ steps.} as a QR factorization: $\bigO(mn\mmn)$. One drawback of the randomized setting is that we cannot fully leverage the efficient updates of $R_{11}^{-1}R_{12}$ or column norms of $M$ since we are not tracking those quantities. Instead, each time we sample a block of columns we must compute its portion of $R_{12},$ $R_{11}^{-1}R_{12},$ and $M$. Assume we have selected $\ell$ columns so far, then this requires applying $\ell < m$ Householder reflectors, $\bigO(m \ell k_b)$, and running backwards substitution with $k_b$ right hand sides, $\bigO(\ell^2 k_b)$. Running Bischof-Stewart pivoting forward costs at most $\bigO(m k_b \min{(m,k_b)})$ (assuming we get $\min{(m,k_b)}$ pivots out of the sample.) The key point is that each time we sample a block of columns the work done on that block is at most $\bigO(m^2 k_b)$.

The overall complexity of \cref{alg:randBSQR} is, therefore, entirely dependent on how many blocks we have to sample, $T_s$, to find $m$ total columns. Unfortunately, this is not so easy to analyze. For uniform sampling it could be quite bad. However, for norm-weighted sampling we observe that across a broad range of test matrices $T_s$ is a small constant when $k_b = m$ (see \cref{fig:sampling}). The cost of randomized Bischof-Stewart pivoting can be cleanly split into two parts: a deterministic part and a non-deterministic part. For norm-weighted sampling the fixed setup cost is $\bigO(mn),$ where we compute the squared column norms and build a Fenwick tree~\cite{Fenwick} to allow for efficient sampling without replacement. The non-deterministic cost\footnote{We do have the, somewhat useless, bound $T_s \le m\frac{n}{k_b}$. The worst case is if at each step we need to sample every column before finding only 1 that satisfies the desired bound. In such a setting the cost of this algorithm actually becomes roughly $\bigO(nm^3)$ and would be more expensive than the deterministic variant.} is $\bigO(T_s m^2 k_b + T_s k_b \log{n}),$ where the $T_s k_b \log{n}$ arises from sampling the blocks with the aforementioned tree. If $R_{12}$ is desired, we pick up a $\bigO(m^2 n)$ post processing cost.
\begin{remark}
In some settings uniform sampling may be viable (see, e.g.,~\cite{cortinovis2025sublinear}), i.e., $T_s$ is a small constant with uniform sampling. In such settings randomized Bischof-Stewart pivoting can compute the desired column subset with only logarithmic dependence on $n$. However, as we see in \cref{fig:largen} this is challenging to observe in practice.
\end{remark}

\subsection{Performance}
There are two natural comparison points for \texttt{randBSQR}: (1) how much more efficient is it than the deterministic variant and the standard Golub-Businger pivoting scheme (accessed via \texttt{DGEQP3}), and (2) how does it compare to other randomized methods applicable to this setting such as randomly-pivoted QR (RPQR)~\cite{cortinovis2026adaptive,epperly2025adaptive} (using the efficient implementation from~\cite{epperly2025adaptive} that leverages rejection sampling as in~\cite{Barthelme2023faster,derezinski2019minimax}). Notably, RPQR only provides guarantees in expectation on column (or row) based approximation error of an underlying matrix using a GKS like scheme. In fact, in $\|\cdot\|_F$ it provides stronger guarantees than \mGKS\ owing to its connections to volume sampling. We do not observe such behavior in practice. In contrast, \cref{alg:randBSQR} provides guarantees on the quality of the submatrix selected and \cref{sec:mGKStheory} turns these into approximation error bounds.

Other relevant work includes~\cite{armstrong2025collect} and~\cite{fakih2025efficient}, neither of which are limited to the case of orthonormal rows. The former guarantees the same output as Golub-Businger QR but avoids excessive updates of columns that will, provably, not be selected. The latter leverages sparse sketches from the right (in contrast to sketches from the left, which are irrelevant when $m\ll n$ and we want $m$ columns) to identify a small number of candidate columns from which a good set can then be selected. While similar ideas could, perhaps, be paired with the ideas of \cref{alg:randBSQR}, the results in~\cite{fakih2025efficient} suggest the cost of sketching limits the speedup compared to what we observe here. Similarly, while the algorithm in~\cite{armstrong2025collect} is more general than \texttt{randBSQR} the speedup over \texttt{DGEQP3} lags what we observe here in the specialized setting. 

First, we consider the speedup achieved by \cref{alg:randBSQR} relative to both our own optimized deterministic version (see \cref{subA:illustrative_performance}) and the built-in QR with column pivoting (i.e., \texttt{DGEQP3} as accessed via MATLAB). \Cref{fig:speedup} shows that for large $n$ \cref{alg:randBSQR} consistently offers a two order of magnitude speedup over the deterministic algorithms when only the selected columns are desired. Even when $R_{12}$ is desired we still achieve an order of magnitude speedup. This does come at some cost. \Cref{fig:quality} shows that the quality of the selected subset does lag that of the version that looks at all the columns in some settings. This is expected given the relatively small number of columns that we consider. Nevertheless, the randomized variant always satisfies the theoretical assurances provided.  

\begin{figure}[t]
\centering
\includegraphics[width=\linewidth]{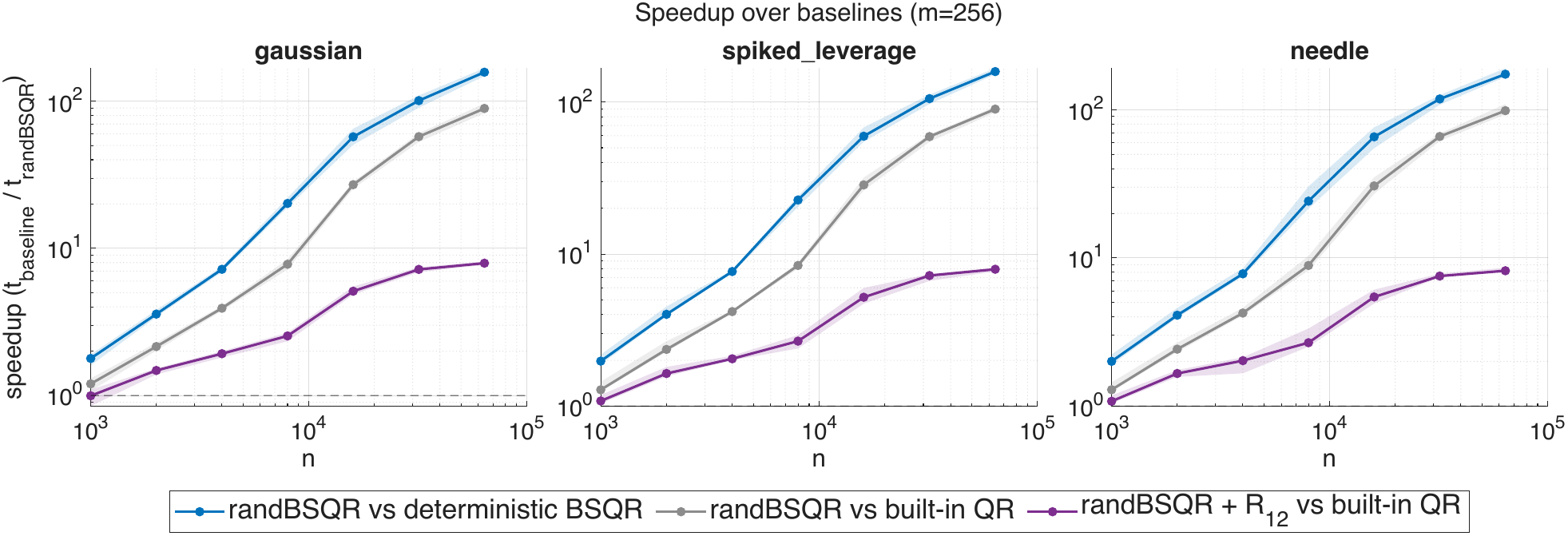}
\caption{Speedup of \cref{alg:randBSQR} relative to deterministic Bischof-Stewart pivoting and a built-in pivoted QR factorization. The solid line represents the median of 20 trials (i.e., distinct instances of the randomly generated matrices) and shaded regions indicate min/max observations. When we do not require $R_{12}$ it is possible to achieve speedups of two orders of magnitude for large $n$; even if $R_{12}$ is desired a significant speedup is achieved. A byproduct of this data is that our C++ MEX implementation of \cref{alg:bsqr} has performance comparable to the built-in QR (always within a factor 2).
}
\label{fig:speedup}
\end{figure}

\begin{figure}[t]
\centering
\includegraphics[width=\linewidth]{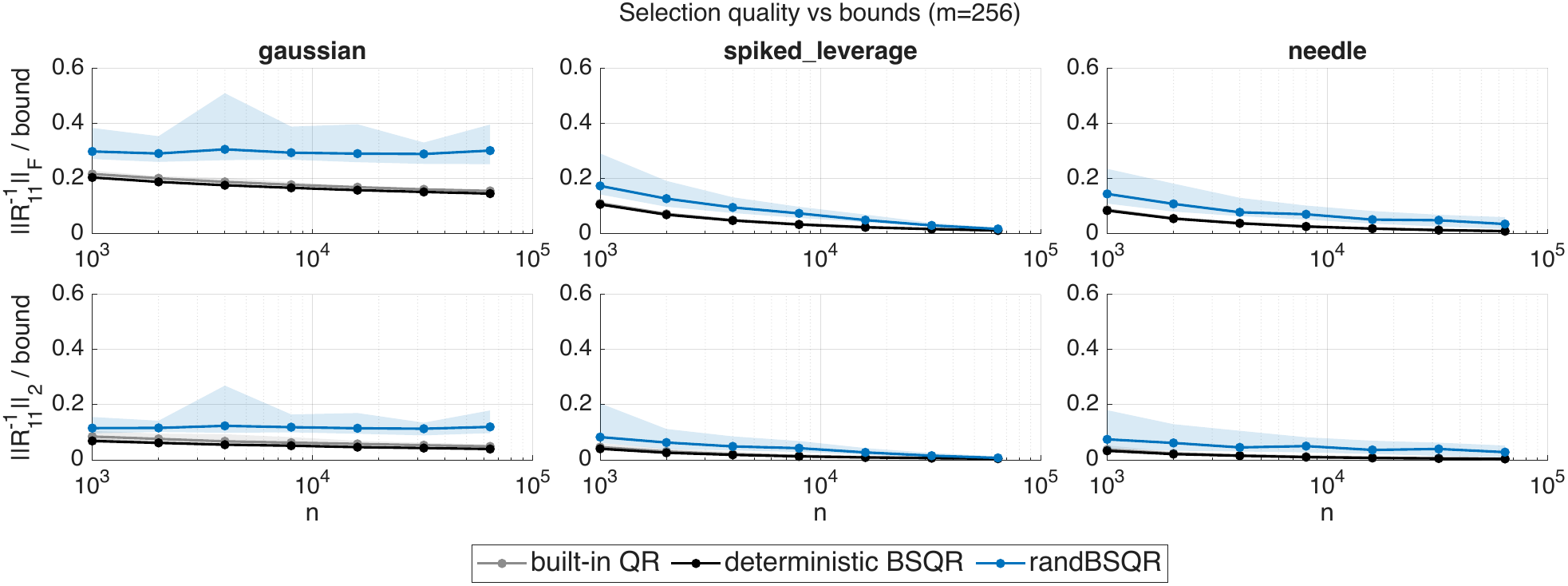}
\caption{Quality of the selected subset from \cref{alg:randBSQR}, deterministic Bischof-Stewart pivoting, and the built-in pivoted QR relative to theoretical bounds from \cref{cor:osinsky}. The solid line represents the median of 20 trials (i.e., distinct instances of the randomly generated matrices) and shaded regions indicate min/max observations. While we observe some degradation relative to the deterministic variant, particularly for the Gaussian case, \cref{alg:randBSQR} always satisfies the bound---as it must. The pivoted QR algorithm of Golub and Businger performs quite well here, albeit without theoretical assurances.
}
\label{fig:quality}
\end{figure}

Now, we explore the relative performance of \cref{alg:randBSQR} and RPQR. First, \cref{fig:RPQRquality} compares the quality of the selected subset. The only theoretical guarantees for RPQR are within the \mGKS\ framework and pertain to the $\|\cdot\|_F$ approximation accuracy---though we expect a reasonable subset to be found. We observe that randomized Bischof-Stewart consistently outperforms RPQR on metrics related to $R_{11}^{-1}$. In fact, RPQR regularly fails to satisfy the theoretical bounds \cref{thm:randBSpivot}, which \cref{alg:randBSQR} guarantees. Notably, this improved performance does not come at the expense of computational cost. \Cref{fig:largen} shows that \cref{alg:randBSQR} is consistently about twice as fast as RPQR in the asymptotic regime. Because of its connections to volume sampling RPQR benefits from more precise theory about its runtime (in terms of the number of columns that need to be sampled); \cref{alg:randBSQR} lacks such assurances. Nevertheless, \cref{fig:largen} suggests that $T_s$ behaves similarly to the sampling requirements for (blocked) RPQR (see~\cite{epperly2025adaptive}).

\begin{figure}[t]
\centering
\includegraphics[width=\linewidth]{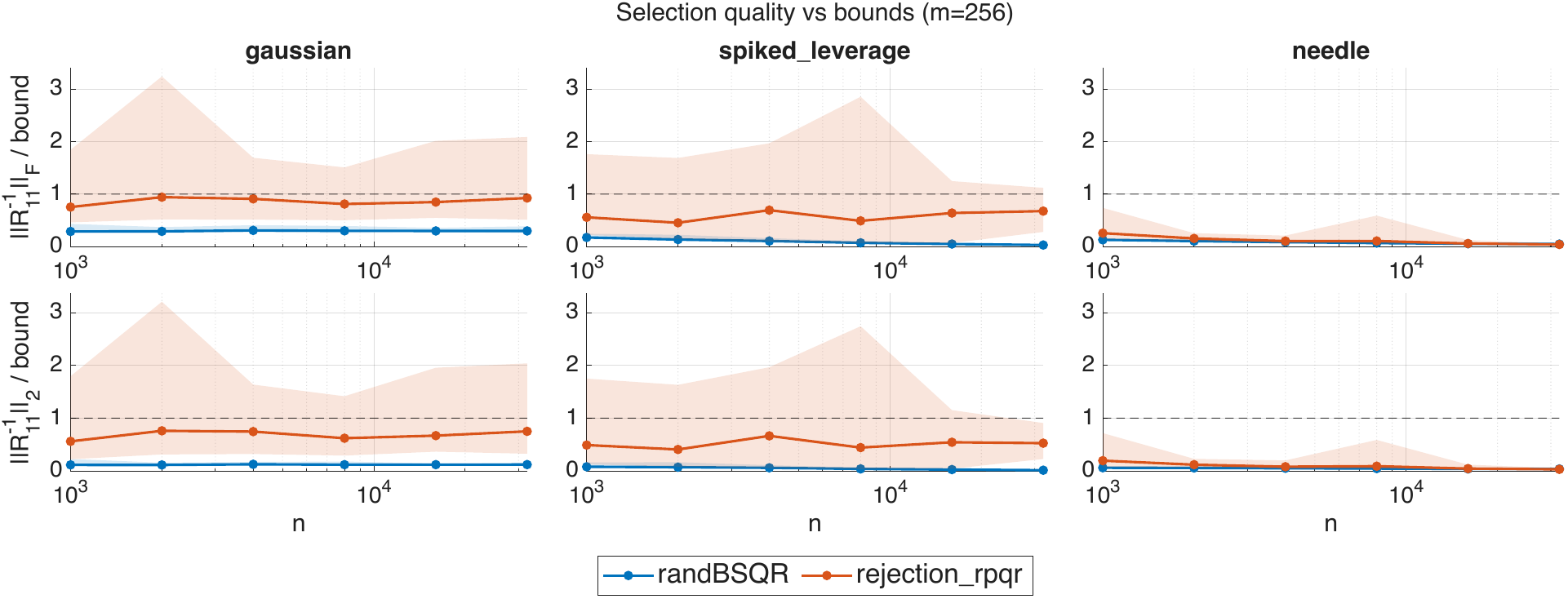}
\caption{Quality of the selected subset from \cref{alg:randBSQR} and RPQR relative to the bound from \cref{cor:osinsky}. The solid line represents the median of 20 trials (i.e., distinct instances of the randomly generated matrices) and shaded regions indicate min/max observations. \Cref{alg:randBSQR} always satisfies its theoretical guarantees; RPQR cannot provide the same assurances. 
}
\label{fig:RPQRquality}
\end{figure}

\begin{figure}[t]
\centering
\includegraphics[width=0.66\linewidth]{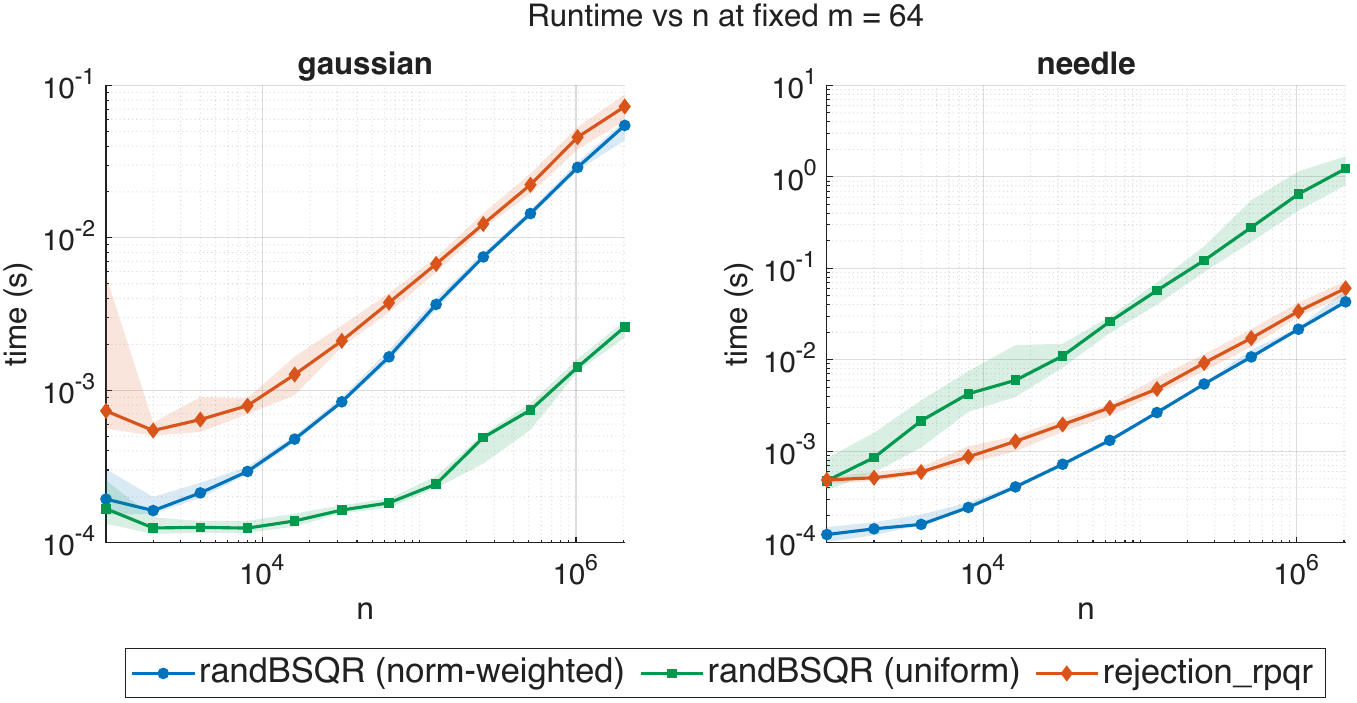}
\caption{Compute time for \cref{alg:randBSQR} with norm-weighted sampling and uniform sampling, and RPQR. The solid line represents the median of 20 trials (i.e., distinct instances of the randomly generated matrices) and shaded regions indicate min/max observations. Notably, we see a persistent, asymptotically stable gap between \cref{alg:randBSQR} and RPQR of around 2. In some settings, uniform sampling can be practically effective. However, without knowledge of the distribution of column norms ahead of time it is virtually impossible to know whether one is in the situation on the left or the right. Norm-weighted sampling has more consistent performance.
}
\label{fig:largen}
\end{figure}

Finally, we consider how the \mGKS\ framework performs when using randomized Bischof-Stewart or RPQR. To accomplish this we consider synthetic matrices with prescribed spectrum and varying right singular vectors. In \cref{fig:mgks} we see that even though $\|R_{11}^{-1}\|$ and, therefore, $R_{11}^{-1}R_{12}$ behave considerably better for randomized Bischof-Stewart, the projection error $\|A-P_\C A\|_2$ is often similar for both methods---though they do differ, with randomized Bischof-Stewart outperforming RPQR, around sharp drops in singular values. Randomized Bischof-Stewart does identify a more stable basis for interpolation (as evidenced by the maximum magnitude entry in $R_{11}^{-1}R_{12}$ being smaller). When an oblique projector is used to determine the low-rank approximation (i.e., up to a permutation from the right our low-rank approximation is $A(\mc,\C)\begin{bmatrix} I & R_{11}^{-1}R_{12}\end{bmatrix}$) randBSQR consistently outperforms RPQR.

\begin{figure}[t]
\centering
\includegraphics[width=\linewidth]{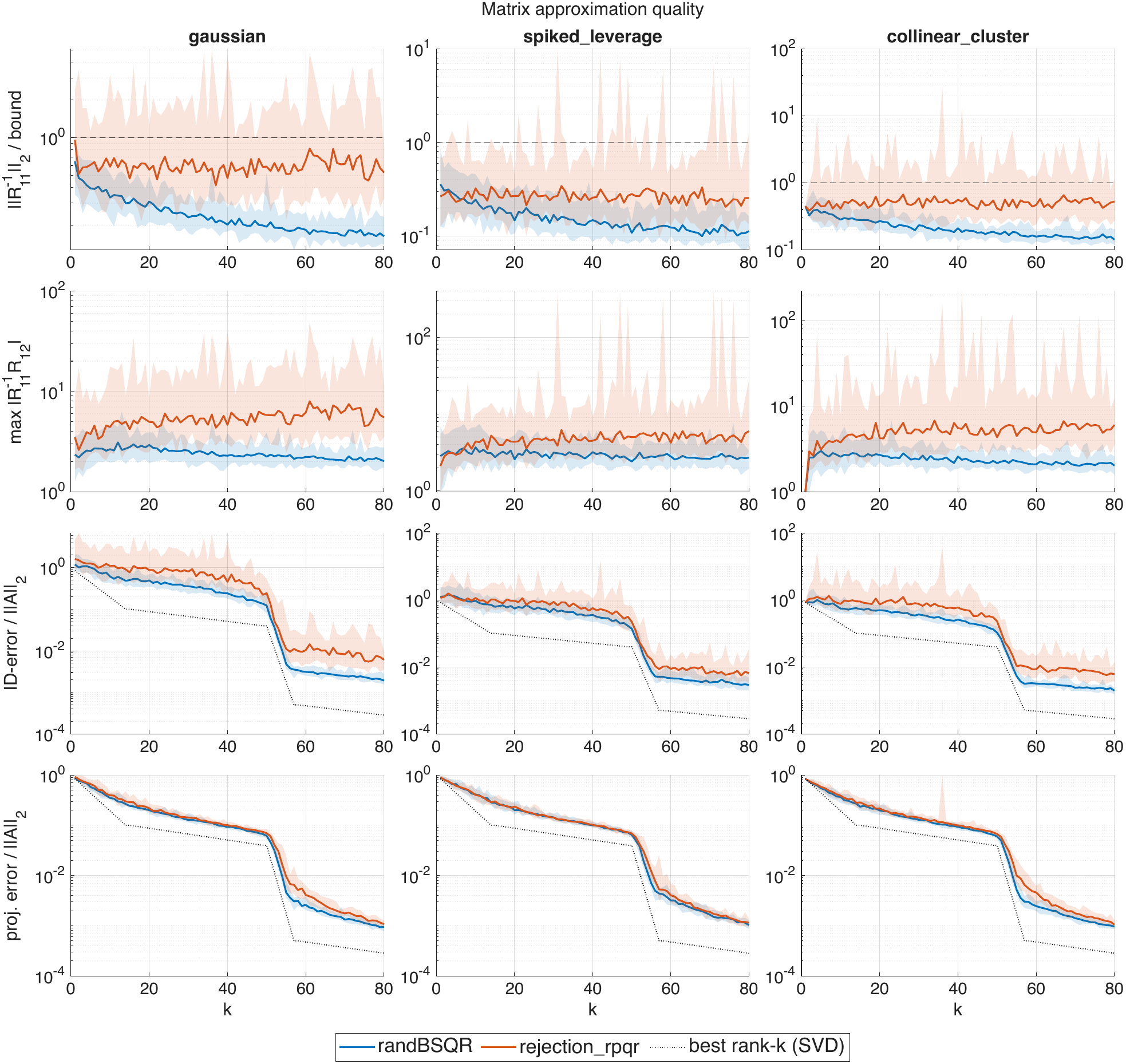}
\caption{Quality of the basis selected from $V_k^T$ in \mGKS\ using randomized Bischof-Stewart and RPQR in terms of $\|R_{11}^{-1}\|_2$ (top) and $\lvert R_{11}^{-1}R_{12}\rvert_{\max}$ (2nd from top). Low-rank approximation error when using an oblique ($\|A\Pi^T - A(\mc,\C)\begin{bmatrix} I & R_{11}^{-1}R_{12}\end{bmatrix}\|_2$) projector (3rd from top) and orthogonal ($\|A-P_\C A\|_2$) projector (bottom). The singular values of the $2000\times 2000$ matrices were explicitly specified and the leading $k$ right singular vectors were drawn from the respective distributions. The solid line represents the median of 20 trials (i.e., distinct instances of the randomly generated right singular vectors) and shaded regions indicate min/max observations.
}
\label{fig:mgks}
\end{figure}

\section{Discussion} 
\label{sec:discussion}	

There are several natural questions raised by this work that we briefly discuss. In addition we highlight some applications where the results could have immediate impact.

\subsection{Improving Frobenius norm approximation bounds}
Owing to its connections to volume sampling~\cite{deshpande2006adaptive,deshpande_rademacher_projectiveclustering,deshpande_rademacher_volumesampling}, in the \mGKS\ framework RPQR yields strong theoretical guarantees~\cite{cortinovis2026adaptive} for Frobenius norm approximation error. A natural question is if Bischof-Stewart satisfies similar ($n$-independent) bounds, i.e., if \cref{thm:mGKS} is loose in $\|\cdot\|_F$. Osinsky shows a greedy strategy can achieve optimal Frobenius norm bounds, but it requires access to the residuals. So, randomness may not be, strictly speaking, necessary. The more pertinent question is if it's possible to identify a good submatrix of $V_k$ but observe poor approximation error; while this is not possible in the two norm, the answer is less clear in the Frobenius norm where it is often possible to ``hide'' error by, e.g., replicating columns.

\subsection{Picking rows and columns}
We have framed the most natural application of Bischof-Stewart pivoting as within the \mGKS\ framework, and a clear use of the selected columns is as part of a low-rank approximation. One may also wish to build a two-sided interpolatory low-rank approximation where columns of $A$ are used to approximate the column space and rows of $A$ are used to approximate the row space---a $CUR$ factorization~\cite{mahoney_cur}. The path to using Bischof-Stewart in this context is clear. We simply run \mGKS\ on $A^T$ to identify rows and use those as part of the factorization. Such a scheme readily admits theoretical bounds~\cite{mahoney_cur,sorensen_embree_deim} based on the quality of just a column (or row) based approximation.

\subsection{Cholesky factorizations}
There are close connections between pivoted QR factorizations and pivoted Cholesky factorizations. In particular between RPQR and randomly pivoted Cholesky~\cite{chen2025randomly,epperly2025embrace}. Therefore, a natural question is what the analogue of Bischof-Stewart pivoting is for Cholesky factorizations---presumably it corresponds to not just pivoting based on (progressively updated) diagonal entries but rather on a combination of diagonal entries and interpolation coefficients. 

\subsection{Matrices without orthonormal rows}
Applying Bischof-Stewart pivoting to general matrices $A$ could lead to improved performance over Golub-Businger pivoted QR in certain settings. However, as noted by Stewart~\cite{stewart1990incremental} the worst case performance is similar. Moreover, the randomized variant cannot be easily extended to this setting as it is unclear what acceptance criterion to use when selecting a column.

There is one setting where Bischof-Stewart is still useful---identifying a well-conditioned square submatrix of a short wide matrix $W\in\R^{m\times n}$ with $m < n$. In this case the best possible conditioning of the submatrix is dictated by the singular values of $W$. While we can no longer prove that Bischof-Stewart can be used to compute a strong rank-revealing factorization, there is a way to identify $m$ columns of $W$ whose minimal singular value is as large as we would get from a strong rank-revealing factorization. Specifically, if we compute the reduced QR factorization $W^T = QR$ and then apply (randomized) Bischof-Stewart pivoting to $Q^T$ it identifies a subset $\C$ of columns of $W$ such that $\smin(W(\mc,\C)) \geq \smin(W)/\sqrt{1+m(n-m)}$. This result follows immediately from \cref{thm:BSpivot} and a singular value variant of Ostrowski's theorem for eigenvalues~\cite{ostrowski1959quantitative}.

\subsection{Applications}
Even with the restriction to orthonormal rows, there are several applications that could immediately benefit from this work. For example, problems in electronic structure theory~\cite{damle2015compressed,damle2018disentanglement}, model reduction~\cite{drmac_qdeim}, and optimal design~\cite{eswar2026Doptimal} all directly rely on pivoted QR factorizations of short wide matrices with orthonormal rows. In such applications it is common to appeal to strong rank-revealing QR factorizations for theoretical guarantees, but to simply use \texttt{DGEQP3} in practice. \texttt{randBSQR} can provide the necessary theoretical assurances while simultaneously providing a means to compute the desired submatrix/columns much more efficiently than is possible with \texttt{DGEQP3}.

\appendix

\section{Bischof and Stewart's modified pivoting criterion, in detail}
\label{secA:BSpivoting}
We present the pivoting strategy of Bischof and Stewart in detail. In addition, we provide experimental results from two implementations to show the performance relative to existing, optimized codes (specifically, \texttt{DGEQP3} leveraging the Apple Accelerate framework). For clarity, we present the algorithm in terms of the linear algebra involved with a focus on how the next pivot is selected---``standard'' approaches would be used to ensure computational efficiency. We make a few remarks about the numerical implementation when relevant.

\subsection{The algorithm}
\label{subA:BSdetails}
Given $A\in\Rmn$ for any $0 \leq i <\mmn$ assume we have the partially computed (full) QR factorization
\begin{equation}
\label{eq:ApivotedQR}
(Q^{(i)})^T A\Pi^{(i)} = \begin{bmatrix}R_{11}^{(i)} & R_{12}^{(i)} \\ & M^{(i)}\end{bmatrix}
\end{equation}
where $Q^{(i)}$ is orthogonal, $R_{11}\in\R^{i\times i}$ is upper triangular, $M\in\R^{(m-i) \times (n-i)}$, and $R_{12}\in\R^{i\times (n-i)}$. To make this well specified, we adopt the convention that given $\Pi^{(i)}$ all other matrices are those computed via the first $i$ steps of a QR factorization without pivoting. By convention $\Pi^{(0)} = I,$ though the final output of the algorithm is essentially independent of the choice of $\Pi^{(0)}$ if no ties arise in the pivot selection. 

Given $\Pi^{(i)}$ we describe how to compute $\Pi^{(i+1)}$. The algorithm then reduces column $i+1$ of $A\Pi^{(i+1)}$ to upper triangular form using a Householder reflector, increments $i$ to $i+1$ and continues until $i = \mmn$ (unless a partial factorization is explicitly desired). The design choice of Bischof-Stewart is for the first $i$ columns of $\Pi^{(i+1)}$ to be those of $\Pi^{(i)}$ and to choose the $i+1$ column of $\Pi^{(i+1)}$ to minimize $\|(R_{11}^{(i+1)})^{-1}\|_F$ subject to this constraint. \Cref{alg:bsqr} provides an algorithmic description of the pivoting scheme; it is written for clarity rather than performance (though it does include the necessary steps to ensure the asymptotic cost is generically $\bigO(mn\mmn)$ and $\bigO(mnk)$ if stopped after $k$ steps). When $m>n$ the algorithm is easily adapted to return a reduced QR factorization if $k = \rank{A}$ by dropping rows of $R$ that are zero.
\begin{remark}
The Bischof-Stewart scheme is greedy, just like the one of Golub and Businger. Once a column is chosen as a pivot it can never be ``removed.'' In contrast, strong rank-revealing algorithms typically require ``swaps'' where previously selected columns in $\Pi^{(i)}$ can be evicted if later determined to have been a poor choice~\cite{gu_srrqr,chandrasekaran_1994,damle2025estimating}. While schemes with swaps can generally achieve better theoretical bounds, in theory the required number of swaps could be large and result in significant computational overhead.
\end{remark}

\begin{algorithm}[t]
\caption{Bischof-Stewart pivoted QR.}
\label{alg:bsqr}
\begin{algorithmic}[1]
\Require $A\in\mathbb{R}^{m\times n}$; steps $r\le \rank{A}$ 
\State Initialize $R=A$; $\pi = \begin{bmatrix} 1 & 2 & \cdots & n\end{bmatrix}$; $i=0$; and $w_j = \begin{bmatrix}\phantom{1} \end{bmatrix}$ and $\rho_j^2 = \|A(\mc,j)\|_2^2$ for $j=1,\ldots,n$ 
\While{$i<r$}
\State $i=i+1$
\State Choose $\displaystyle \ell \in \argmin_{j\in\{i,\ldots, n\vert \rho_j>0\}} \frac{1+\|w_j\|_2^2}{\rho_j^2}$ \Comment{By convention $\|w_j\|_2 = 0$ when $i=1$} \label{line:BSpiv}
\State Swap $R(\mc,i)$ and $R(\mc,\ell),$ $\pi(i)$ and $\pi(\ell),$ $w_{i}$ and $w_\ell,$ and $\rho_{i}^2$ and $\rho_{\ell}^2$
\State $v_i =$ \Call{house}{$R(i\mc m,i)$} \Comment{$(I-2v_iv_i^T)R(i\mc m,i) = \pm \|R(i\mc m,i)\|_2e_1$}
\State $R(i\mc m,i\mc n) = R(i\mc m,i\mc n) - 2v_i(v_i^TR(i\mc m,i\mc n))$
\State $\rho_j^2 = \rho_j^2 - R(i,j)^2$ for $j=i+1,\ldots,n$
\State $\displaystyle w_j = \begin{bmatrix} w_j - \alpha_j w_i\\ \alpha_j\end{bmatrix}$ for $j=i+1,\ldots,n$, where $\alpha_j = R(i,j)/R(i,i)$ \label{line:wup}
\EndWhile
\State \Return $\pi$ (encoding $\Pi$), $R$, and the stored reflectors $v_i$ (an implicit $Q$) such that $A\Pi = QR$, an $r$ step partial QR factorization of $A$.
\end{algorithmic}
\end{algorithm}

Diving into \cref{alg:bsqr} we make two key observations. First, as illustrated in \cref{eq:RinvF_update} the pivoting strategy in \cref{line:BSpiv} due to Stewart~\cite{stewart1990incremental} explicitly minimizes the growth in the Frobenius norm of $(R_{11}^{(i)})^{-1}$ from one step to the next. Since our convention at the first step is that $\|w_j\|_2 = 0$, the first pivot selected by this scheme is the same as the first one selected by the classical Golub and Businger scheme. In the unlikely event of a tie at any step any reasonable choice is fine, e.g., picking the smaller index or picking one of the minimizing indices at random.

The second key observation is that \cref{line:wup} is actually maintaining the columns of $(R_{11}^{(i)})^{-1}R_{12}^{(i)}$ (see~\cite[eq. 2.3]{stewart1990incremental}). This is important to ensure the desired computational scaling---a na\"{\i}ve scheme that recomputes these vectors at each step would pick up an extra factor of $k$ in the scaling. In fact, the update in \cref{line:wup} can be manipulated to show that we can also cheaply update the squared norm of $w_{j}$ as
\begin{equation}
\label{eq:wnormdowndate}
\|\hat{w}_{j}\|_2^2 = \|w_{j}\|_2^2 - 2\alpha_jw_{j}^Tw_{i} + \alpha_j^2(\|w_i\|_2^2+1).
\end{equation}
While useful, this formula is not needed to ensure the correct asymptotic scaling.
\begin{remark}
When updating the $\rho_j^2$ and the $\|w_{j}\|_2^2$ one should be careful of cancellation. This issue can be avoided by recomputing norms in certain cases without changing the asymptotic complexity of the algorithm. See~\cite{drmac2008failure} for details regarding $\rho_j^2;$ $\|w_j\|_2^2$ can be treated analogously. Our implementations contain these safeguards.
\end{remark}

\subsection{Illustrative performance}
\label{subA:illustrative_performance}
Though it has the same asymptotic complexity, the Bischof-Stewart strategy is more computationally intensive than the method of Golub and Businger (as practically manifest in the LAPACK routine \texttt{DGEQP3}). We provide two optimized implementations of Bischof-Stewart pivoted QR---one C++ mex file usable from MATLAB~\cite{matlab} and one in Julia~\cite{julia,Julia-2017}. The experimental setup mirrors that in \cref{sec:randBS}.

Informally, in the roughly square ($m\approx n$) or short and wide setting ($m \ll n$) performance can easily be within a factor 2 of the optimized LAPACK routine \texttt{DGEQP3}.\footnote{The tall thin case ($n \ll m$), while an important setting for column selection, is somewhat less interesting in the context of this manuscript. However, Bischof-Stewart QR can have performance closer to \texttt{DGEQP3} in this case. The cost of updating the current $R_{11}^{-1}R_{12}$ factor used as part of the pivot selection becomes inconsequential relative to the cost of applying the Householder reflectors.} In fact, in the short wide setting where it is often desirable to compute the interpolation coefficients $R_{11}^{-1}R_{12}$ the gap is even smaller---a standard pivoted QR requires a subsequent triangular solve to get these coefficients whereas Bischof-Stewart pivoting can get them at minimal additional cost since that quantity is related to the pivoting selection scheme. \Cref{fig:BSperfA} quantifies these observations in Julia. MATLAB experiments may be found in the main text.

\begin{remark}
In \cref{fig:BSperfA} it would be reasonable to expect that performance is insensitive to matrix structure. This is true for Bischof-Stewart pivoting across all three test cases and \texttt{DGEQP3} for the matrices with iid Gaussian entries and orthonormal rows. However, \texttt{DGEQP3} slows down for the ill-conditioned ($\kappa \approx 10^{10}$) case and, therefore, Bischof-Stewart looks comparatively better. A hypothesis for why we observe this behavior is that \texttt{DGEQP3} has some inefficiency (in terms of blocking) when norms have to be recomputed rather than using the updating formulas. Our implementation of Bischof-Stewart pivoting contains the same safeguards that result in norms being recomputed explicitly, but batches those computations to maintain efficiency.
\end{remark}

\begin{figure}[t]
\centering
\begin{minipage}[t]{0.5\textwidth}
	\vspace{0pt}
	\centering
	\includegraphics[width=\linewidth]{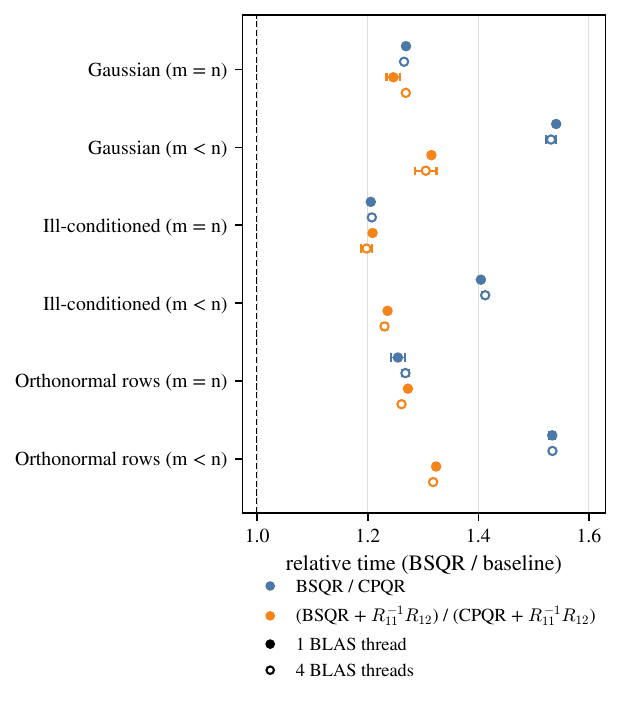}
\end{minipage}
\hfill
\begin{minipage}[t]{0.45\textwidth}
	\vspace{0pt}
	\centering
	\caption[Relative performance of Bischof-Stewart pivoting to \texttt{DGEQP3} in Julia.]{Relative performance of Bischof-Stewart pivoting to \texttt{DGEQP3} in Julia. Each data point amalgamates data across a range of problem sizes. In the square case this includes $m$ from 64 to 512, and in the short wide case this includes $m$ from 32 to 512 with $n$ ranging from $2m$ to $10m$. For each random seed, runtime ratios are aggregated by taking the geometric mean over all problem sizes in the group. The marker shows the geometric mean of these per-seed aggregates and the whiskers their full range. Each runtime is a median of many repeated runs (using BenchmarkTools in Julia), leading to the low variation. As expected, the performance gap typically closes when $R_{11}^{-1}R_{12}$ is desired since its computation is essentially a byproduct of the Bischof-Stewart scheme (when using \texttt{DGEQP3} an extra triangular solve is needed). This is particularly the case for the $m<n$ case. In all cases performance is comfortably within a factor 1.6.\\
	}
	\label{fig:BSperfA}
\end{minipage}
\end{figure}

\section*{Acknowledgements}

We would like to thank participants of the Simons workshop on complexity and linear algebra for many insightful discussions related to this work. In particular, discussions with Daniel Kressner, Ethan Epperly, Robert Webber, Ilse Ipsen, Mark Embree, Laura Grigori, Yuji Nakatsukasa, Mark Fornace, Alex Townsend, and Michael Lindsey sharpened our thinking of the relationship between \mGKS\ and Osinsky's results. Implementations were produced using Anthropic's Claude Code~\cite{claude}. As was code to run numerical experiments and generate plots. OpenAI Codex~\cite{codex} was used for code review and early prototyping. Deterministic implementations were validated against a carefully audited oracle implementation (albeit one that was far less performant). Both versions were extensively tested to assess correctness; these tests are included in the repository. AI tools (Anthropic's Claude and OpenAI's ChatGPT) were used to aid in proofreading parts of the manuscript and to provide detailed reviews of drafts that guided revisions.

\bibliographystyle{siam}
\bibliography{gks}

\end{document}